\documentclass[12pt, reqno]{amsart}
\usepackage{mathrsfs}
\usepackage{amssymb}
\usepackage{cite}
\usepackage[all]{xy}
\usepackage{graphicx}
\usepackage{textcomp}
\usepackage{charter}
\usepackage[vflt]{floatflt}
\usepackage{mathtools}
\usepackage{enumerate}
\usepackage{rotating}
\usepackage{pdflscape}
\usepackage{comment}
\usepackage{fullpage}
\usepackage{bm}
\usepackage{url}
\usepackage[pdftex,colorlinks]{hyperref}
\usepackage{tikz,animate}
\numberwithin{equation}{section}
\hypersetup{
    colorlinks=true,
    citecolor=blue, 
     linkcolor=black,     
    urlcolor=black, 
}

\usepackage{tabularx}

\usepackage{enumitem} 
\newtheorem{theorem}{Theorem}[section]
\newtheorem{lemma}[theorem]{Lemma}

\newtheorem{corollary}[theorem]{Corollary}
\theoremstyle{definition}

\theoremstyle{remark}
\newtheorem{remark}[theorem]{Remark}
\numberwithin{equation}{section}

\newcommand{\B}{\mathcal{B}}
\newcommand{\D}{\mathbb{D}}
\newcommand{\h}{\mathcal{H}}
\newcommand{\m}{\mathcal{M}}

\newcommand{\T}{\mathbb{T}}
\newcommand{\5}{\theta H^2}
\newcommand{\8}{K_\theta}

\newcommand{\ho}{\overline{H_0^2}}

\newcommand{\li}{L^\infty}
\newcommand{\s}{\mathcal{L}at}
\usepackage{tikz-cd}
\newtheorem{question}{Question}

\begin{document}

\title[Invariance Preserving Conjugation and Linear Operators]{Invariance Preserving Conjugation and Linear Operators on the Hardy Space}

\author[A. Chattopadhyay and S. Jana]{Arup Chattopadhyay and Supratim Jana}

\address{A. Chattopadhyay : Department of Mathematics, Indian Institute of Technology Guwahati, Guwahati, 781039, India.}

\email{arupchatt@iitg.ac.in, 2003arupchattopadhyay@gmail.com.}

\address{S. Jana : Department of Mathematics, Indian Institute of Technology Guwahati, Guwahati, 781039, India.}

\email{supratimjana@iitg.ac.in, suprjan.math@gmail.com.}



\begin{abstract}
We study linear and anti-linear operators on the Hardy space $H^2$ in connection with the Beurling lattice. First, in contrast to Sarason's result \cite{S1}, we show that no non-zero anti-linear operator leaves every shift-invariant subspace invariant. Then we completely characterize the class of conjugations, isometries, and invertible linear operators preserving the shift-invariant subspace structure in terms of weighted composition operators induced by disk automorphisms--combined with the canonical conjugation, isometric multiplication, and invertible Toeplitz operators, respectively. We also establish the existence and non-existence of conjugations that map shift-invariant subspaces to coinvariant subspaces. Finally, we construct a counterexample that rules out an analogous factorization for bounded linear operators preserving shift-invariance of closed images.
\end{abstract}

\maketitle

\section{{Introduction and Preliminaries}}

Conjugations have been extensively studied in recent years due to their significant role in operator theory, matrix analysis, and other branches of science. The origins of this subject can be traced back to mathematical physics. Following the introduction of \emph{Complex Symmetric} operators by Garcia-Putinar in \cite{GP1, GP2}, this topic has become an active area of research and has gained considerable attention. We begin by introducing the notation and terminology that will be used throughout this article.

Let $\h$ be a (separable) complex Hilbert space and $\B(\h)$ denote the algebra of bounded linear operators acting on $\h$. An anti-linear map $C: \h \rightarrow \h$ is said to be a \emph{conjugation} if it is an isometric involution, that is,
\begin{align*}
    & \langle C f, C g\rangle = \langle g, f \rangle \quad \text{and} \quad  C^2 h = h \qquad \text{for all $f,g ,h\in \h$}.
\end{align*}
An operator $T\in \B(\h)$ is said to be $C$-symmetric if $CTC=T^*$ or equivalently $CT = T^*C$, where $T^*$ is the adjoint operator of $T$. This operator $T$ is called complex symmetric if it is $ C$- symmetric with respect to some conjugation. The characterization of conjugations possessing certain properties is an important and interesting problem that has attracted considerable attention and motivated extensive research; see, for example, \cite{BarTim, CD, DyPP, GW, Gu, JKKL14, LZ,  NSW15, NSW14, NSW20, ZL}.

Let $\mathbb{D}:=\{ z \in \mathbb{C}: |z| < 1 \}$ denote the unit disk in the complex plane and $\mathbb{T}:=\{z\in\mathbb{C}:|z|=1\}$ its boundary. Also let $L^2:=L^2(\mathbb{T})$ be the Hilbert space of square-integrable functions, and $L^\infty:=L^\infty(\mathbb{T})$ denote the von Neumann algebra of essentially bounded functions with respect to the normalized Lebesgue measure on $\mathbb{T}$.

The Hardy-Hilbert space of analytic functions on $\D$ is denoted by $H^2$, which is also identified (via non-tangential boundary limits \cite{FricainMas, GMR, NKN, VVP}) as a closed subspace of $L^2$ consisting precisely of those functions whose negative Fourier coefficients vanish. Let $H^\infty$ denote the Banach algebra of bounded analytic functions on $\D$. A bounded analytic function $\theta$ is called inner if its boundary value is unimodular almost everywhere on $\mathbb{T}$.

The Toeplitz operator $T_\phi: H^2 \rightarrow H^2$, introduced by Brown-Halmos \cite{BH}, is defined by
$$T_\phi (f)=P(\phi f)$$ where $P: L^2 \rightarrow H^2$ is the Szeg\"o projection and $\phi\in L^\infty$. The elementary Toeplitz operator $T_z$ is commonly denoted by $S$ and is called the \emph{(forward) shift} operator, which plays a fundamental role in function theory. Beurling in \cite{AB} (1949) provided a complete characterization of all $ S$-invariant subspaces of $H^2$, which are precisely of the form $\5,$ where $\theta$ is an inner function. In other words, the lattice of $S$-invariant subspaces is given by $$\mathcal{L}at(S) = \{ \theta H^2: \theta \text{ is an inner function } \}. $$ Consequently, the lattice of $S^*$-invariant (coinvariant) subspaces is $$\s(S^*) = \{ H^2 \ominus \theta H^2 : \theta \text{ is an inner function } \}.$$ These subspaces are called model spaces, denoted by $\8$ and determined by $$ \8 = H^2 \cap \theta \overline{H^2_0}.$$

A foundational result of Sarason (\cite{S1}, 1966) asserts that \emph{any bounded linear operator $T: H^2 \rightarrow H^2$ leaving every shift-invariant subspace invariant is an analytic Toeplitz operator}. A natural, missing piece in invariant subspace theory is understanding its anti-linear analog. Surprisingly, we observe that:

\begin{theorem}\label{5m1}
    There is no non-zero anti-linear map $C: H^2 \rightarrow H^2$ such that $C(M) \subset M$ for every $M\in \s(S)$.
\end{theorem}

However, we have the canonical conjugation map $J: H^2 \rightarrow H^2$ defined as 
\begin{equation}\label{5e1}
    (Jf)(z) = f^*(z) = \overline{f(\bar{z})}
\end{equation}
which sends the shift-invariant subspaces $\5$ to $\theta^* H^2$. In other words, $J$ preserves the shift-invariance structure. Therefore, the following natural structural questions arise: 

\begin{question}\label{5q1}
    \emph{Is $J$ rigid in this respect, or does the same be possessed by a much wider family of conjugations?} 
\end{question}

\begin{question}\label{5q2}
    \emph{Can a global conjugation mediate between Beurling subspaces and model spaces by mapping shift-invariant subspaces onto coinvariant ones?}
\end{question}

To address both questions, our approach relies on the fundamental results of Brown-Halmos \cite{BH}, Sarason \cite{S1}, and Crofoot \cite{C}. Together, these results provide the key ingredients needed for our argument and ultimately lead to the following conclusions.

In response to Question \ref{5q1}, we have the following characterization.

\begin{theorem}\label{5th1}
    Let $C: H^2 \longrightarrow H^2$ be a conjugation on $H^2$ and let $M$ be any shift-invariant subspace. Then,  $C(M)$ is shift-invariant if and only if $$ C=  \lambda U_{a, w} J \quad \text{ such that $wa = \bar a$ },$$ where $\lambda$ and $w$ are unimodular constants and $a\in \D$.
\end{theorem}

Here, $U_{a, w}: H^2 \rightarrow H^2$ is the unitary operator, given by 
\begin{equation}\label{5u}
    (U_{a, w}f) (z) = \frac{\sqrt{1-|a|^2}}{1-\bar{a}z} f\bigl(w\frac{a-z}{1-\bar{a}z}\bigr) \qquad \text{ for $a \in \mathbb{D}$, and $w\in \T$ }.
\end{equation}
And as an answer to Question \ref{5q2}, we obtain the following characterization.

\begin{theorem}\label{5m2}
    There is no conjugation $C$ in $H^2$ that holds $C(\m(S)) = \m(S^*)$, where $\m(S)$ and $\m(S^*)$ are as described in Section 3.
\end{theorem}

Now moving back to the setting of linear operators, it is a well-known fact that a Toeplitz operator $T_\phi$ is an isometry if and only if its symbol $\phi$ is an inner function $\theta$. In particular, this characterization establishes a fundamental connection between inner functions and isometric Toeplitz operators on the Hardy space $H^2$. Moreover, $T_\theta(M)\in \s(S)$ for every $M\in \s(S)$. Now, if we remove "Toeplitz" from the mentioned fact, we arrive at the following question: 

\begin{question}\label{5q3}
    What are the linear isometries $T: H^2 \rightarrow H^2$ such that $T(M) \in \s(S)$ for every $M\in \s(S)$?
\end{question}

We indeed succeed in characterizing all such isometries, and arrive at the following conclusion.

\begin{theorem}\label{5m3}
    Let $T: H^2 \rightarrow H^2$ be a bounded linear operator such that $T(M) \in \s(S)$ whenever $M\in \s(S)$. Then $T$ is an isometry if and only if $$ T=T_\eta U_{a,w}$$ where $\eta$ is an inner function and $U_{a,w}$ is the (weighted-composition) unitary operator defined in Equation \eqref{5u}.
\end{theorem}

The remainder of the article is organized as follows. Section 2 is devoted to the necessary background and develops the framework required to establish Theorem \ref{5m1} and \ref{5th1}. Section 3 concerns the characterization of conjugations on $H^2$ that map a shift-invariant subspace $\eta H^2$ onto a coinvariant subspace $K_\theta$. In addition, it describes certain Hankel-type intertwining conjugations on the model space $\8$. Section 4 contains a systematic approach to proving Theorem \ref{5m3} together with some additional observations on the linear operator preserving Beurling lattice $\s(S)$ and an illustrative counterexample.


\section{Invariance Preserving Conjugations} 

First, we recall a result of Sarason \cite{S1} which will play a crucial role in the proof of Theorem \ref{5th1}. For the convenience and clarity of the readers, we present a detailed proof below.

\begin{lemma}\label{5l1}
    Let $T: H^2 \rightarrow H^2$ be a bounded linear operator that leaves every shift-invariant subspace invariant, that is, $$T(\5) \subset \5 \quad \text{ for every inner function $\theta.$ }$$  Then $T$ is an analytic Toeplitz operator.
\end{lemma}

\begin{proof}
Let $M$ be a closed subspace of $H^2$. Then $M$ is an invariant subspace of $T$ if and only if $M^\perp$ is invariant under $T^*$, that is, $T(M)\subset M \iff T^*(M^\perp) \subset M^\perp.$ Consequently, in this context $$T^* (\8) \subset \8$$ for every inner function $\theta$.

It is well known that each $a\in \D$, the function $\frac{1}{1-\bar{a}z} (:= v_a, \text{ say })$ is an eigen vector corresponding to the eigen value $a$, that is, $S^* v_a = a v_a$. Now, observe that $\mathbb{C} \{v_a\} = K_{\mu_a} $, where $\mu_a$ is the disk automorphism $\frac{z-a}{1-\bar{a}z}$, also a single Blaschke factor with zero at $a$.

Since $K_{\mu_a}$ is one-dimensional, we have $ T^* v_a = \lambda_a v_a, $ for every $a\in \D.$ Therefore, for every $a\in \D$, we have
$$ T^* S^* (v_a) = \lambda_a a v_a = S^*T^* (v_a) . $$

Since $v_a$ is a reproducing kernel for $H^2$ at the point $a$, $\bigvee_{a\in \D} \{ v_a \} = H^2.$ Thus, $T^*$ commutes with $S^*$, and equivalently $TS=ST$. Hence, the conclusion follows from Theorem $7$ of \cite{BH}.
\end{proof}

We now prove its anti-linear analogue.
\vspace{0.1in}

{\bf Proof of Theorem \ref{5m1}}
\vspace{0.1in}

Let $C: H^2 \rightarrow H^2$ be an anti-linear operator such that $C(\5) \subset \5$ for every inner function $\theta$. Now, we choose the single Blaschke factor $\mu_a:= \frac{a-z}{1-\bar{a}z}$ for $a\in \D$. Then, we have $$\mu_a H^2 = \{ f\in H^2: f(a)=0 \}.$$ 
Now, $$C(\mu_a H^2) \subset \mu_a H^2 \implies C(f)(a) = 0, \qquad \text{for every $f\in \mu_a H^2$} .$$  
So, for a given $h\in H^2$, we always have $h-h(a) \in \mu_a H^2$. Then, we have 
\begin{align*}
    & C(h-h(a)1)(a) = 0\\
    & \implies C(h)(a) - \overline{h(a)} C(1)(a) = 0\\
    & \implies C(h)(a) = \overline{h(a)} C(1)(a) \qquad \\
    & \implies g_1(a) = \overline{h(a)} g_2(a) \qquad \text{(consider $C(h) = g_1 \in H^2$ and $C(1) = g_2 \in H^2$)}.
\end{align*}
Therefore, the above calculation shows that the action of $C$ on $H^2$ is determined by $C(1)$. Since $C$ is a non-zero map, there is $h_0 \in H^2$ and $a_0 \in \D$ such that $C(h_0)(a_0) \neq 0$. So, neither $C(1)(a_0) = g_2 (a_0)$ nor $\overline{h_0(a_0)}$ is zero. Therefore, there is an open set $U$ of $\D$ containing $w_0$ such that $g_2(a) \neq 0$ for $a\in U$. In that case, we have $$ \overline{h_0} = \frac{g_1}{g_2} \qquad \text{on $U$} .$$ But in this case, $\overline{h_0}$ is co-analytic on $U$ whereas $\frac{g_1}{g_2}$ is analytic. This occurs if and only if $h_0$ is a constant function. For $h_0 = z^{n} ~ (n\geq1)$, $C(h_0) = 0$ and applying them on the derived identity $C(h)(a)= \overline{h(a)} C(1) (a)$, we have $$ \overline{a^n} C(1) (a) = 0 \qquad \text{ for every $a\in \D$. }$$ This implies $C(1)(a) = 0$, for all non- zero $a\in \D$. Therefore, by continuity, we have $C(1) = 0$ on $\D$, and hence the anti-linear operator $C = 0$. \hfill $\square$

\vspace{0.1in}

Next we prove our first main result.
\vspace{0.1in}

{\bf Proof of Theorem~\ref{5th1}}

\vspace{0.1in}

Let $C: H^2 \rightarrow H^2$ be the conjugation such that $C(\5) \in \s(S),$ for every inner function $\theta$. Then, it is readily seen that $C(\s(S))=\s(S).$ 
       
For $M\in \s(S)$, we have 
\begin{align*}
    & S(M) \subset M \quad \implies SC(CM)) \subset C(CM) \quad \implies CSC(CM)\subset CM.
\end{align*}
According to our hypothesis, $C(M)\in \s(S)$, and since $CSC$ is a linear operator, by the Lemma~\ref{5l1}, we have $CSC=T_\phi,$ for some $\phi \in H^\infty$. 
       
But $||CSC h|| = ||h|| \text{ for all $h\in H^2.$ } $ So, $T_\phi$ is an isometry and thus $\phi = \theta,$ for some inner function $\theta.$

Moreover, $(CSC)^* = CS^*C,$ and $Ker(CS^*C)$ is one dimensional, so is the dimension of $(Ker T_{\bar{\theta}})$. Therefore, inner function $\theta$ is a single Blaschke factor $w\frac{a-z}{1-\bar{a}z}$, for some $a\in \D$ and $w \in \T$ (disk-automorphism). Thus, 
\begin{equation} \label{5e3}
    CSC = T_{\mu_{a, w}} \qquad \text{ where } \qquad \mu_{a, w} = w\frac{a-z}{1-\bar{a}z}.
\end{equation}
Now, for $a\in \D,$ we have the unitary operator $U_a: H^2 \rightarrow H^2$, defined by 
$$ (U_af) (z) = \frac{\sqrt{1-|a|^2}}{1-\bar{a}z} f\bigl(\frac{a-z}{1-\bar{a}z}\bigr),$$ 
due to Crofoot \cite{C}. Observe that $\mu_{a, w} $ is an involution when $w=1$, and $U_a^2 = Id \implies U_a^{-1}= U_a^* = U_a$.

Now, we construct the unitary operator, displayed in Equation \eqref{5u}, in the following manner. For a constant $w\in \T$, there is a natural unitary operator $$R_w: H^2 \rightarrow H^2  \qquad\text{ defined by } \qquad R_w (f)(z) = f(wz) ,$$ known as the rotation operator. Consequently, the operator $U_{a,w}$ can be expressed as $$U_{a, w} = U_a R_w,$$ making it a unitary operator and therefore $$ U_{a,w}^* f(z) = U_{a, w}^{-1} f(z) = R_{\bar{w}} U_a f(z) = \frac{\sqrt{1-|a|^2}}{1- \bar{w} \bar{a}z} f \big( \bar{w} \frac{wa-z}{1- \bar w \bar a z} \big) = U_{aw, \bar w} f(z).$$
       
The following direct computation yields:
\begin{align*}
    U_{a, w} S f (z) & = U_{a, w} (zf(z)) \\
    & = \mu_{a, w} \frac{\sqrt{1-|a|^2}}{1-\bar{a}z} f\bigl(w\frac{a-z}{1-\bar{a}z}\bigr)\\
    & = T_{\mu_{a, w}} U_{a,w} f(z).
\end{align*}
This implies $U_{a, w} S U_{a, w}^* = T_{\mu_{a, w}} $.



Moreover, if there is some other unitary operator $V: H^2 \rightarrow H^2$ such that $VSV^*= T_{\mu_{a, w}},$ then the unitary operator $W= U_{a, w}^* V$ satisfies 
\begin{align*}
    & WS = U_{a, w}^* V S = U_{a, w}^* T_{\mu_{a, w}} V = S U_{a, w}^* V = SW.
\end{align*}
Therefore, by [Theorem 7, \cite{BH}] $W=T_\eta$ for some $H^\infty$ function $\eta$. But the fact that $W$ is unitary implies that $\eta$ is a unimodular constant. Thus, $V= U_{a, w}$ is unique up to a unimodular constant (proving the uniqueness of unitary equivalence of the shift).
\vspace{0.1in}
       
Now, form Equation \eqref{5e3}, we have
\begin{align*}
    CSC =  T_{\mu_{a, w}} & \iff CS = T_{\mu_{a, w}} C \\
    & \iff CSJ = T_{\mu_{a, w}} CJ \\
    & \iff CJS = U_{a, w} S U_{a, w}^* CJ \qquad \text{ (since $JS=SJ$) }\\
    & \iff (U_{a, w}^* SJ) S = S (U_{a, w}^* CJ)\\
    & \iff U_{a, w}^* CJ = T_\psi
\end{align*}
Since $C$ and $J$ are two conjugations on $H^2$, their multiplication $CJ (:= U$, say) is a linear operator. Moreover, since $U^* = JC$, we have $UU^*=I=U^*U$, showing $U$ is unitary. Therefore, $U_{a, w}^*CJ$ as well as $T_\psi$ are unitary operators, and hence $\psi$ is a unimodular constant $\lambda$. 

Therefore, from the above we have $ U_{a, w}^*CJ = \lambda \iff C= \lambda U_{{a, w}} J,$ which is an anti-linear solution of the Equation \eqref{5e3}.


Also, since $C$ is an involution, $C^2 =I$, which immediately yields
\begin{align*}
    &  \lambda U_{{a, w}} J \lambda U_{{a, w}} J =I \\
    & \iff J U_{a, w} J = U_{a, w}^* \\
    & \iff U_{\bar a, \bar w} = U_{a, w}^* \qquad \text{ since $JU_{a, w} J = U_{\bar a, \bar w}$ }\\
    & \iff \frac{\sqrt{1-|a|^2}}{1-az} f(\bar{w} \frac{\bar{a} - z}{1-az}) = \frac{\sqrt{1-|a|^2}}{1-\bar w \bar a z} f(\bar{w} \frac{{wa} - z}{1-\bar w \bar az}) \quad \text{for arbitrary $f\in H^2$}.
\end{align*}
It is not difficult to see that the above equality holds if and only if $wa = \bar a$. Hence, the required conjugation we derive is $C = \lambda U_{a, w} J$ such that $wa = \bar a$.


\vspace{0.1in}

Conversely, if $C= \lambda U_{a, w} J$, then 
\begin{align*}
    C(\5) & = \lambda U_{a, w} J(\5)\\
    & = U_a (\eta H^2) \qquad \text{(where $\eta=\theta^*$, an inner function)}\\
    & =( \eta \circ \mu_{a, w})  U_a(H^2)\\
    & = (\eta \circ \mu_{a, w}) H^2.
\end{align*}
And it is straightforward to verify that $\eta \circ \mu_{a, w}$ is an inner function. The proof is therefore complete. \hfill $\square$

Note that in the relation $wa = \bar a$, different values of $w$ correspond to different values of $a$. For example, 
\begin{align*}
    & w=1 \qquad \text{corresponds to} \qquad a\in (-1, 1)\\
    & w=-1 \qquad \text{corresponds to} \qquad a= ir \text{ for } r\in (-1, 1).
\end{align*}
In more general terms, if $w= e^{i t}$ for $t\in (-\pi, \pi]$, and $wa = \bar a$, then $a = r e^{-it/2}$, where $0\leq r<1$.

\vspace{0.1in}

As an immediate consequence, we have the following observation.

\begin{corollary}
    A conjugation $C$ on $H^2$ preserves coinvariance property, that is, $C(\8) \in \s(S^*)$ if and only if $C=  \lambda U_{a, w} J$, where $\lambda$ and $w$ are unimodular constants and $a\in \D$.
\end{corollary}

These conjugations readily commute with the orthogonal projection $P$, but not all commute with $S$.

\begin{corollary}
    The only conjugation on $H^2$ that commutes with $S$ is $J$ (unique up to multiplication by a unimodular constant). 
\end{corollary}
         
\begin{proof}
    The proof is straightforward from the last part of the above proof with $a=0$.
\end{proof}

The above discussion leads to the following remark.

\begin{remark}
    $J$ is the unique conjugation that preserves the shift-invariance and commutes with the shift operator.
\end{remark}

\section{Invariance to Coinvariance via Conjugations}

It is evident that when the model space $\8$ is finite-dimensional, then the impossibility of a conjugation $C$ on $H^2$ with $C(\8) = \eta H^2$ occurs. So, we consider a sub-collection of the lattice of $S$ and $S^*$ as:
\begin{align*}
    \mathcal{M}(S) & = \{ M\in \s(S): dim(M) = \infty = dim(M^\perp)\}\\
    & = \{ M\in \s(S): M = \theta H^2 \text{ where } \theta \text{ is not a finite Blaschke product }\}\\
    & \hspace{2in} \text{ equivalently }\\
    & \{ M^\perp \in \s(S^*): M^\perp = \8 \text{ with } dim(\8) = \infty \} = \mathcal{M}(S^*).
\end{align*}

Now, let $M\in \mathcal{M}(S)$ and $N \in \mathcal{M}(S^*)$. Then $M=\eta H^2$ and $N=\8$ (with $\eta, \theta $ are not finite Blaschke). Now, consider the orthonormal basis as follows:
\begin{align*}
    H^2 = \eta H^2 \oplus K_\eta & = \bigvee_{n=0}^\infty\{ \eta z^n \}  ~\oplus~ \bigvee_{m=0}^\infty{\{e_m\}}\\
    & = \bigvee _{n=0}^\infty \{ f_n \} ~ \oplus ~ \bigvee_{m=0}^\infty \{ \theta z^m \} = \8 \oplus \5.
\end{align*}
Define \begin{equation}\label{5eq}
    C_\eta^\theta: H^2 \rightarrow H^2
\end{equation} as $$ C_\eta^\theta(\alpha_n \eta z^n) = \overline{\alpha_n} f_n \quad \text{ and } \quad C_\eta^\theta(\beta_m e_m) = \overline{\beta_m} \theta z^m ,$$
where $\alpha_n, \beta_m$ are constants. Then, it is straightforward to verify that $C_\eta^\theta$ is a conjugation on $H^2$ satisfying $C_\eta^\theta(\eta H^2) = \8$ and $C_\eta^\theta(K_\eta) = \5$. 

Taking $\eta=\theta$ (not a finite Blaschke), the existence of a conjugation $C$ such that $C(\5)=\8=({\5})^\perp$ is guaranteed. But \emph{is the existence of a global conjugation possible for all $M\in \mathcal{M}(S)$?}

\begin{theorem}
    There is no conjugation on $H^2$ that satisfies $C(M) = M^\perp$ for all $M\in \mathcal{M}(S)$.
\end{theorem}

\begin{proof}
If possible, let $C: H^2 \rightarrow H^2$ be a conjugation such that $C(M) = M^\perp,$ for every $M\in \mathcal{M}(S)$. Also, let $\eta$ be a non-constant inner function. Then, both $\theta H^2$ and $\eta\theta H^2$ belong to $\mathcal{M}(S)$. 
            
Since $\theta \eta H^2 \subset \theta H^2$, then it follows that $\8 \subset K_{\eta\theta}$. But
\begin{align*}
\theta \eta H^2 \subset \theta H^2  \implies C(\theta \eta H^2) \subset C(\theta H^2)  \implies K_{\eta \theta} \subset K_\theta,
\end{align*} which is an impossibility.
\end{proof}

Now, we investigate the possibility of a conjugation in $H^2$ such that $C(M)\in \m(S^*)$ for every $M\in \m (S)$. To do so, we make use of the following Lemma, which is a slight modification of Lemma \ref{5l1}.

\begin{lemma}\label{5l2}
    Suppose $T: H^2 \rightarrow H^2$ is a bounded linear operator such that $T(\5) \subset \5$ for every inner function $\theta$ that is not a finite Blaschke product. Then $T$ is an analytic Toeplitz operator.
\end{lemma}

\begin{proof}
    Let $\theta$ be an infinite Blaschke product, and $\eta$ be a singular inner function. Also, let $B_n$ be any finite Blaschke product. Then both $B_n \5$ and $B_n \eta H^2$ are in $\mathcal{M}(S)$, and they are invariant under the action of $T$.
            
    Then,
    \begin{align*}
    & T(\overline{span}(B_n \5, B_n \eta H^2)) \subset \overline{span}(B_n \5, B_n \eta H^2)\\
    & \implies T(B_n H^2) \subset B_n H^2 \quad \text{ because }\\
    & \overline{span}(B_n \5, B_n \eta H^2) = gcd(B_n \theta, B_n \eta) H^2 = B_n H^2 \quad \text{ since }(gcd(\theta, \eta)=1).
    \end{align*}

Therefore, $T(M) \subset M,$ for every $M\in \s(S)$, and hence by Lemma~\ref{5l1}, $T$ is an analytic Toeplitz operator.
\end{proof}

{\bf Proof of Theorem \ref{5m2}:}

\vspace{0.1in}

Let $M\in \mathcal{M}(S)$. Then $S(M) \subset M$. If possible, let $C: H^2 \rightarrow H^2$ be a conjugation such that $C$ maps $\mathcal{M}(S)$ onto $\mathcal{M}(S^*).$

Then, $CSC(CM) \subset (CM)$ and $C(M) \in \mathcal{M}(S^*).$ Therefore, by Lemma~\ref{5l2}, we have
    \begin{align*}
    & CSC = T_{\bar{\phi}} \qquad \text{ for some $\phi \in H^\infty$ }.
    \end{align*}

But, $CSC$ is an isometry on $H^2$, so is $T_{\bar{\phi}}$, and then $\bar{\phi}$ is an inner function. Also, $\phi \in H^\infty$ implies $\phi$ is a unimodular constant. This is an impossibility; hence, it is proved. \hfill $\square$

\vspace{0.1in}

As a direct consequence, we have the following observation.

\begin{corollary}\label{5c1}
    There is no conjugation $C$ on $H^2$ that satisfies the Hankel-type equation $S^*C = CS$.
\end{corollary}

\begin{proof}
    $CS = S^*C \iff CSC = T_{\bar{z}}$, and the rest follows from the above proof.
\end{proof}

However, instead of $H^2$, if the conjugation $C$ in Corollary \ref{5c1} is considered on $\8$, then the Hankel type intertwining equation $ S_\theta^* C = CS_\theta $ possesses non-trivial solution, one of them is $C_\theta: \8 \rightarrow \8$, given by $C_\theta(f)=\theta\bar{z}\bar{f}$. Note that it is just the compression of the conjugation $C_\theta: L^2 \rightarrow L^2$, given by $C_\theta(f)=\theta\bar{z}\bar{f}$, which admits the following properties. 
\begin{enumerate}[label=(\roman*)] 
    \item $C_\theta (\8)=\8 \quad \text{and}\quad C_\theta( {\8}^\perp ) = {\8}^\perp$
    \item ${(A_\phi^\theta)}^* = C_\theta A_\phi^\theta C_\theta \quad \text{and} \quad {(D_\phi^\theta)}^* = C_\theta D_\phi^\theta C_\theta,$
\end{enumerate}
where $A_\phi^\theta$ is a truncated Toeplitz operator (TTO) (introduced by Sarason \cite{S3}), and $D_\phi^\theta$ is a dual truncated Toeplitz operator (DTTO) (by Ding-Sang \cite{DS}; also see \cite{CKLP3} by C\^arama et al. and \cite{G} by Gu for characterization).

\vspace{0.1in}

Recently, Gu-Ma in \cite{GM}, classified the intertwining linear operators $$I(S_\theta, S_\theta^*):= \{  X\in \mathcal{B}(\8): S_\theta^* X = X S_\theta \}$$ as well as $I(S_\theta^*,S_\theta)$ (an independent study of $I(S_\theta, S_\theta^*)$ has been done by B.-Bala-Panja-Sarkar in \cite{BBPS}). 

\vspace{0.1in}
      
Now, a natural question that arises from them is: \emph{Find the anti-linear solution to the intertwining equations $S_\theta^* X = X S_\theta$ and $S_\theta X = X S_\theta^*$}.
      
\vspace{0.1in}
      
Exploiting the complex-symmetric nature of the compressed shift $S_\theta$, we obtain the following observation.

\begin{theorem}
    The anti-linear operator satisfying both the operator equation $S_\theta^* X = X S_\theta$ and $S_\theta X = X S_\theta^*$ is precisely of the form $\lambda C_\theta$, for $\lambda \in \mathbb{C}$.
\end{theorem}

\begin{proof}
Let $X$ be the anti-linear operator such that $S_\theta^* X = X S_\theta$. Then, we have
    \begin{align*}
    S_\theta^* X = X S_\theta & \iff C_\theta S_\theta C_\theta X = X S_\theta\\
    & \iff S_\theta (C_\theta X) = (C_\theta X) S_\theta.
    \end{align*}
Therefore, by Sarason's commutant lifting theorem [\cite{S2}, Theorem 1], we have $C_\theta X = A_\phi^\theta$ for some $\phi\in H^\infty$. Thus, the anti-linear solution to $S_\theta^* X = X S_\theta$ is $$ C_\theta A_\phi^\theta \quad \text{  for $\phi \in H^\infty$. }$$

Proceeding similarly, we get that the anti-linear solution to the operator equation $S_\theta X = X S_\theta^*$ is $$ A_\phi^\theta C_\theta \quad \text{  for $\phi \in H^\infty$. }$$
         
Therefore, due to our hypothesis, we have: 
\begin{align*}
    & \{ C_\theta A_\phi^\theta: \phi\in H^\infty \} \cap \{ A_\phi^\theta C_\theta: \phi\in H^\infty \}\\
    & = \{ C_\theta A_\phi^\theta: \phi\in H^\infty \} \cap \{C_\theta A_{\bar{\phi}}^\theta: \phi\in H^\infty \} \\
    & = \{ \lambda C_\theta: \lambda\in \mathbb{C} \}.
\end{align*} Hence the conclusion follows.
\end{proof}

As an immediate consequence, we have the following.
\begin{corollary}
    If $X$ in the above theorem is a conjugation, then $X=\lambda C_\theta$, where $\lambda$ is a unimodular constant, is the only solution.
\end{corollary}

\subsection*{Special Remark}

We end this section with an observation on certain conjugation operators in $L^2$.  We introduce the notion of an \emph{orthogonally-preserving} conjugation, that is, a conjugation $C: \h \rightarrow \h$ satisfying $C(M) = M^\perp$ for a closed subspace $M$ of $\h$; for example $C_\theta^\theta$ in $H^2$ (see \eqref{5eq}), $V(f)=\bar{z}\bar{f}$ in $L^2$. We then relate this notion to the $ M_z$-conjugations, $M_z$-commuting conjugations introduced by C\^amara et al. in \cite{CKLP1, CKLP2, CKP}, and obtain an interesting result.
     
\begin{remark}
Let $C$ be a $M_z$-conjugation in $L^2$ that orthogonally preserves $H^2,$ then $C= \lambda V$, for some unimodular constant $\lambda$. Also, there is no $M_z$-commuting conjugation in $L^2$ that orthogonally preserves $H^2$.
\end{remark}

\begin{proof}
Due to Theorem~2.2 and Theorem~2.4 of \cite{CKLP1}, we have $M_z C_1 = C_1 M_{\bar{z}}$ if and only if there exists $\phi \in \li$ with $|\phi|=1$ such that $C_1(f) = M_\phi \bar{f}$ and $C_2 M_z = M_z C_2$ if and only if there exists $\psi\in \li$ with $|\psi|=1$ such that $C_2(f) = M_\psi f^*$ (where $f^*(z) = \overline{f(\bar{z})}$ for $f\in L^2$).

\vspace{0.1in}

Let us now consider $C_1(H^2) = \ho.$ Then $\phi = M_\phi \bar{1} = C_1(1) \in \ho.$ Also, $C(\ho) = H^2$ implies that for $n\geq1,$ $\langle \phi {z}^n, \bar{z}\rangle = 0 \implies \langle \phi, \bar{z}^{n+1} \rangle =0$ thus $\phi \in \vee\{\bar{z}\}$. Now $C_1^2 = I$ gives that $\phi = \lambda \bar{z},$ for some $\lambda \in \mathbb{C},$ with $|\lambda| =1$.

\vspace{0.1in}

Next, consider $C_2(H^2) = \ho.$ Then similarly as above $\psi \in \ho.$ And for $m\geq 0,$ we have $\langle \phi z^m, 1\rangle = 0 \implies \langle \phi, \bar{z}^m \rangle = 0 \implies \phi=0,$ which is a contradiction.
\end{proof}

\section{Invariance Preserving Linear Operators}

This section is devoted to studying the linear operators on $H^2$ in connection with the Beurling lattice $\s(S)$. To begin with, we first recall that an isometric Toeplitz operator on $H^2$ maps a shift-invariant subspace to a shift-invariant subspace, preserving its shift-invariant structure. This gives rise to the natural Question \ref{5q3}, which leads to the characterization Theorem \ref{5m3}, and the proof follows as below.

\vspace{0.1in}

{\bf Proof of Theorem~\ref{5m3}}

Let $T: H^2 \rightarrow H^2$ be an isometry such that $T(M) \in \s(S)$ whenever $M\in \s(S)$. Therefore, 
\begin{equation*}
    T(H^2) = \eta H^2 \qquad \text{for some inner function $\eta$.}
\end{equation*}
Since $T$ is an isometry with the range $\eta H^2$, we have $T^*T = I$ and $TT^* = P_{\eta H^2}$, where $P_{\eta H^2}: L^2 \rightarrow \eta H^2$ is the orthogonal projection. 

Now, consider the linear operator $$T_0 = T_\eta^* T.$$ Then, $T(H^2) = H^2$ and $||T_0f|| = ||T_{\bar \eta} T f|| = ||Tf|| = ||f|| $, showing that $T_0$ is an isometry, so, $T_0^* T_0 = I$. Moreover, $$T_0T_0^* = T_{\bar \eta} T T^* T_\eta = T_{\bar \eta} P_{\eta H^2} T_\eta = I,$$ implying $T_0$ is a unitary operator. Also, we have 
\begin{equation}\label{54e1}
    T = P_{\eta H^2} T = T_\eta T_{\bar \eta} T = T_\eta T_0 .
\end{equation}
Now, let $f\in M (\s(S))$. Then $Tf = \eta T_0 f \in T(M)$. Since, according to our hypothesis, $T(M) \in \s(S)$, we  have $$zTf = z\eta T_0f \in T(M),$$ and $T(M) = T_\eta T_0 (M) = \eta T_0(M)$ implies $$ z T_0 f \in T_0 (M) .$$ Therefore, $T_0(M)$ is a shift-invariant subspace of $H^2$. Thus, $$ST_0(M) \subset T_0 (M) \implies T_0^* S T_0 (M) \subset M$$ for every $M\in \s(S)$. Then, by Lemma \ref{5l1}, we have $T_0^* S T_0 = T_\phi$ for some $\phi\in H^\infty$. Since, $T_0^* S T_0$ is an isometry and $Ker(T_0^* S^* T_0)$ is one-dimensional, we have $\phi = \bar w \frac{z-aw}{1-\bar w \bar a z} = \mu_{aw, \bar w}$, for some $a\in \D$ and $w\in \T$. Therefore, as in the argument in the proof of Theorem \ref{5m1}, by the uniqueness of the unitary equivalence of the shift $S$, we have $T_0 = U_{aw, \bar w}^* = U_{a, w}$. Hence, from Equation \eqref{54e1}, we have $T = T_\eta U_{a,w}$. 

Conversely, if $T= T_\eta U_{a, w}$, then for every $\5\in \s(S)$, we have $$ T(\5) = T_\eta U_{a, w} (\5) = \eta (\theta \circ \mu_{a, w}) H^2 \in \s(S) .$$ This completes the proof. \hfill $\square$

Note that if we consider the operator $T= U_{a, w} T_\theta$, then also $T(M) \in \s(S)$ for every $M \in \s(S)$. But this operator $T$ can be written as $$ T= T_{\theta \circ \mu_{a, w}} U_{a, w},$$ and $\theta \circ \mu_{a, w}$ is an inner function. Therefore, the statement of Theorem \ref{5m3} holds true.

However, for an arbitrary bounded linear operator $T$ acting on a Banach space $X$, $T(M)$ may not be closed even if $M$ is a closed subset of $X$. But imposing some condition(s) on $T$ may allow $T(M)$ to be closed. One such condition is invertibility of the operator $T$; that is, if $T$ is an invertible operator on $X$, then $T(M)$ is closed in $X$ for every closed subset $M$. Therefore, in the context of our study of the linear operators preserving the Beurling lattice, we have the following observation.

\begin{theorem}\label{54m1}
    Let $T: H^2 \rightarrow H^2$ be a linear invertible operator such that $T(M) \in \s(S)$, for every $M\in \s(S)$. Then, $T$ has the form $$ T= T_g U_{a, w},$$ where $g\in \mathcal{G}(H^\infty)$, the group of invertible elements of $H^\infty$.
\end{theorem}

\begin{proof}
    Let $T$ be an invertible linear operator on $H^2$ such that $T(M)$ is shift-invariant for every $M \in \s(S)$.
    Then \begin{align*}
        ST(M) \subset T(M) \iff T^{-1} S T (M) \subset M.
    \end{align*}
    By Lemma \ref{5l1}, we have
    \begin{equation}\label{541}
        T^{-1}ST = T_\phi \qquad \text{ for some $\phi \in H^\infty$. }
    \end{equation}
    
    So, for every $n \in \mathbb{N}$, we have $$(T^{-1}ST)^n = T^{-1} S^n T = T_{\phi^n}.$$
    Therefore, form the above relation, we have $$|| \phi^n|| = || T_{\phi^n}|| \leq ||T^{-1} ||T|| \qquad \text{ for every $n\in \mathbb{N}$}.$$ This implies $|\phi| \leq 1$ almost everywhere on $\T$.

    Since both $T$ and $T^{-1}$ are bounded operators, $$||f|| = ||T T^{-1} f|| \leq ||T|| ||T^{-1}f|| \implies ||T^{-1}f|| \geq \frac{1}{||T||} ||f||,$$ and similarly, $$ ||Tf|| \geq \frac{1}{||T^{-1}||} ||f||.$$
    
    So, for every $f\in H^2$ and $n\in \mathbb{N}$, we have, $$ || T_{\phi^n}f|| = || T^{-1} S^n T f||\geq \frac{1}{||T||} || S^n T f|| = \frac{1}{||T||} ||T f|| \geq \frac{1}{||T|| ||T^{-1}||} ||f||.$$
    This shows that, $|\phi| \geq 1$ almost everywhere on $\T$. Thus, $\phi$ is an inner function. Moreover, since $T^* S^* {T^*}^{-1}$ has codimension $1$, so is true for $T_\phi^*$. Therefore, $\phi (z) = \mu_{a, w}(z) = w \frac{a-z}{1-\bar a z}$, for $|w| =1$ and $a\in \D$.

    From Equation \eqref{541}, we have $T^{-1}S T = T_{\mu_{a, w}}$, where $T$ is an invertible operator. Also, from the previous discussion, we have the weighted composition operator $U_{a, w}$ satisfying the operator equation, that is, $U_{a, w}^{-1} S U_{a, w} = T_{\mu_{a, w}} \iff U_{a, w}^{-1} S  =  T_{\mu_{a, w}}U_{a, w}^{-1}$. Therefore,
    \begin{align*}
        T^{-1}S T = T_{\mu_{a, w}} & \iff ST = T T_{\mu_{a, w}}\\
        & \iff STU_{a,w}^{-1} = T T_{\mu_{a, w}}U_{a,w}^{-1}\\
        & \iff S(TU_{a,w}^{-1}) = (TU_{a, w}^{-1}) S
    \end{align*}
    Thus, by the Brown-Halmos theorem, we have $TU_{a, w}^{-1} = T_g$, for some analytic function $g$; so, $T = T_g U_{a, w}$. But, since $ T(H^2) = H^2 $, we have $ T_g U_{a, w} (H^2) = H^2$. And since $U_{a,w} (H^2) = H^2$, the above reduces to $T_g (H^2) = H^2$, which is true if and only if $g\in \mathcal{G}(H^\infty)$. This completes the proof.
\end{proof}

In the above theorem, we considered $T: H^2 \rightarrow H^2$ to be invertible. But if we relax this condition to $T: H^2 \rightarrow range(T)$ being invertible, where $range(T)$ is a closed subspace of $H^2$, then under the invariance-preserving hypothesis, we obtain the following observation.

\begin{corollary}
    Let $T: H^2 \rightarrow H^2$ be a bounded linear operator such that $range(T)$ is closed and $T: H^2 \rightarrow range(T)$ is invertible. Then for every $M\in \s(S)$, $T(M) \in \s(S)$ if and only if $$ T = T_\eta T_g U_{a, w},$$ where $\eta$ is an inner function and $g\in \mathcal{G}(H^\infty)$.
\end{corollary}

\begin{proof}
    We combine the proofs of Theorems \ref{5m3} and \ref{54m1} to prove the statement.

    Firstly, since $T(M)\in \s(S)$, for every $M\in \s(S)$, we have $T(H^2) = \eta H^2$, for some inner function $\eta$. Now, consider the linear operator $$T_0 = T_\eta^*T.$$ Then $||T_0 h|| = || T_\eta^*T h|| = ||T h||$, shows that $T$ is injective. And $T_0(H^2) = T_\eta^* T(H^2) = H^2$ shows that $T$ is surjective of $T_0$. Therefore, $T_0$ is an invertible operator on $H^2$. 

    Now, proceeding with the proof of Theorem \ref{54m1} for the operator $T_0$, we get $T_0 = T_g U_{a, w}$, where $g\in \mathcal{G}(H^\infty)$. Hence, $T= T_\eta T_g U_{a, w}$.
\end{proof}

Now, moving back to the starting point, Sarason's result, Lemma \ref{5l1}, we observed that for the analytic Toeplitz operator $T_\phi$ and every inner function $\theta$, we have $$T_\phi (\theta H^2) \subset \phi_{inner}\theta H^2,$$ where $\phi = \phi_{inner} \phi _{outer}$ is the inner-outer factorization. And more particularly $$\overline{T_\phi (\theta H^2)} = \phi_{inner}\theta H^2.$$ In similarity with Question \ref{5q3}, we raise the following question.

\begin{question}\label{5q4}
    Let a bounded operator $T: H^2 \rightarrow H^2$ be such that $\overline{T(M)} \in \s(S)$ for every $M\in \s(S)$. Does this imply $T = T_\phi U_{a, w}$, where $\phi \in H^\infty$?
\end{question}

Unlike Theorem \ref{5m3}, the answer to this question is negative. The condition $T = T_\phi U_{a, w}$ with $\phi\in H^\infty$ is a sufficient condition, not a necessary one. Indeed, $$ \overline{T_\phi U_{a, w} (\theta H^2)} = \overline{T_\phi (\theta \circ \mu_{a, w} H^2)} = \phi_{inner}(\theta \circ \mu_{a, w}) H^2 \in \s(S),$$ for every inner function $\theta$. To show that this operator fails to be a necessary condition, we construct a counterexample as follows. 

Any element $f\in H^2$, can be written as $$f(z) = \sum_{n=0}^\infty a_n z^n = \sum_{m=0}^\infty a_{2m}z^{2m} + \sum_{k=0}^\infty a_{2k+1}z^{2k+1}.$$ Define an operator $T: H^2 \rightarrow H^2$ by $$ (Tf)(z) = T\big(\sum_{n=0}^\infty a_n z^n\big) = \sum_{m=0}^\infty a_{2m}z^{m}.$$ 
Then, $$TS^2 f (z) = T\big(\sum_{n=0}^\infty a_n z^{n+2}\big) = \sum_{m=0}^\infty a_{2m}z^{m+1} = STf(z),$$ that is $TS^2 = ST$.
Now, let $M\in \s(S)$. Then, 
\begin{align*}
    S^2 (M) \subset M & \implies TS^2 (M) \subset T(M)\\
& \implies ST (M) \subset T(M) \\
& \implies S(\overline{T(M)}) \subset \overline{T(M)} \qquad \text{ by continuity of $S$ }.
\end{align*}
But this operator $T$ is not an analytic Toeplitz operator because $\{ z^{2k+1} : k \geq 0 \} \subset Ker T$,  whereas kernel of an analytic Toeplitz operator contains only $\{0\}$.



\section*{{{Acknowledgment}}}

We thank Dr. Pratyoosh Kumar for reading the initial manuscript and providing several comments that led us to extend the article.
A. Chattopadhyay is supported by the Core Research Grant (CRG), File No.: CRG/2023/004826, by the Science and Engineering Research Board(SERB), Department of Science \& Technology (DST), Government of India. S. Jana gratefully acknowledges the support provided by IIT Guwahati and the Visiting Scientist position at ISI Bangalore, Government of India.

\vspace{0.2in}

{\bf AI Declaration:} AI has not been used to generate any result.

\bibliographystyle{abbrv}
   
\bibliography{blg}

\end{document}